\documentclass[11pt]{amsart}
\usepackage{graphicx}
\usepackage{amsmath} 
\usepackage{amsthm,amsfonts,amssymb,mathrsfs,amscd,amstext,amsbsy} 
\usepackage{epic,eepic} 
\usepackage{yfonts}
\usepackage{paralist,enumerate}
\usepackage[all]{xy}
\usepackage{hyperref}
\hypersetup{colorlinks}

\newtheorem{theorem}{Theorem}[section]

\newtheorem{lemma}[theorem]{Lemma}

\newtheorem{theo}[theorem]{Theorem}

\newtheorem{definition}[theorem]{Definition}
\newtheorem*{Definition*}{Definition}

\def\qed{\hfill \ifhmode\unskip\nobreak\fi\quad\ifmmode\Box\else$\Box$\fi\\ }

\begin{document}

\title[Torus actions on oriented manifolds of generalized odd type]{Torus actions on oriented manifolds of generalized odd type}
\author{Donghoon Jang}
\address{Department of Mathematics, Pusan National University, Pusan, Korea}
\email{donghoonjang@pusan.ac.kr}
\thanks{Donghoon Jang is supported by Basic Science Research Program through the National Research Foundation of Korea (NRF) funded by the Ministry of Education (2018R1D1A1B07049511).}
\begin{abstract}
In \cite{LS}, Landweber and Stong prove that if a closed spin manifold $M$ admits a smooth $S^1$-action of odd type, then its signature $\mathrm{sign}(M)$ vanishes. In this paper, we extend the result to a torus action on a closed oriented manifold with generalized odd type.
\end{abstract}

\maketitle

\section{Introduction}

$\indent$

The existence of a non-trivial group action on a manifold gives certain restrictions on the manifold and one of them is a characteristic class. Atiyah and Hirzebruch prove that if a closed spin manifold admits a non-trivial circle action, then its $\hat{A}$-genus vanishes \cite{AH}. Hattori generalizes the result to $\textrm{spin}^c$-manifolds \cite{H}. It is shown later that $\hat{A}$-genus vanishes if an oriented manifold with finite second and fourth homotopy groups admits an $S^1$-action \cite{HH}, \cite{HH3}.

In this paper, we discuss the vanishing of $L$-genus of an oriented manifold, that is, the signature of the manifold. The $L$-genus is the characteristic class belonging to the power series $f(x)=\frac{\sqrt{x}}{\tanh \sqrt{x}}$. The signature of an oriented manifold $M$ is the index of the signature operator on $M$. The Atiyah-Singer index theorem states that the $L$-genus of an oriented manifold $M$ is equal to the signature of $M$ \cite{AS}. Kawakubo and Uchida prove that if a closed oriented manifold $M$ admits a semi-free $S^1$-action with $\dim(M^{S^1})<\frac{1}{2}\dim M$, then the signature of $M$ vanishes \cite{KU}. Li and Liu generalize the result to a so-called prime action \cite{LL}. For a vanishing result on the signature of a manifold with a finite group action, see \cite{E} for instance.

Let $M$ be an orientable manifold. Introduce a Riemannian metric on $M$. A spin structure on $M$ is an equivariant lift $P$ (called a principal $Spin(n)$-bundle) of the oriented orthonormal frame bundle $Q$ (called the principal $SO(n)$-bundle) over $M$ with respect to the double covering $\pi: Spin(n) \to SO(n)$.

Let the circle act on a spin manifold $M$. Then the action lifts to an action on the principal $SO(n)$-bundle $Q$. The action is called \textbf{of even type}, if it further lifts to an action on the principal $Spin(n)$-bundle $P$. The action is called \textbf{of odd type}, if it fails to lift to an action on $P$.

Given an action of a Lie group $G$ on a manifold $M$, denote $M^G$ by the set of points fixed by the $G$-action on $M$, i.e., $M^G=\{p \in M|g \cdot p=p, \forall g \in G\}$. If $H$ is a subgroup of $G$, then define the set $M^H$ of points fixed by the $H$-action in the same way.

Given a circle action on a spin manifold $M$, as a subgroup of $S^1$, $\mathbb{Z}_2$ also acts on $M$. The $S^1$-action on $M$ is of even type if and only if each connected component of the set $M^{\mathbb{Z}_2}$ has codimension congruent to 0 modulo 4. Similarly, the $S^1$-action on $M$ is of odd type if and only if each connected component of $M^{\mathbb{Z}_2}$ has codimension congruent to 2 modulo 4. For this, see \cite{AH}.

Now consider a circle action on a manifold $M$. Since $M$ need not allow a spin structure, we use the latter equivalent definition to define an action of even type and an action of odd type. The $S^1$-action is called \textbf{of even type}, if each connected component of $M^{\mathbb{Z}_2}$ has codimension congruent to 0 modulo 4 and \textbf{of odd type}, if each connected component of $M^{\mathbb{Z}_2}$ has codimension congruent to 2 modulo 4. As before, $\mathbb{Z}_2$ acts on $M$ as a subgroup of $S^1$. In \cite{HH2}, H. Herrera and R. Herrera adapt these alternative definitions for circle actions on oriented manifolds.

Landweber and Stong prove that a closed spin manifold admitting a circle action of odd type must have vanishing $L$-genus \cite{LS}.

\begin{theo} \label{t11} \cite{LS} If a closed spin manifold $M$ admits a smooth $S^1$-action of odd type, then its signature $\mathrm{sign}(M)$ vanishes. \end{theo}

In this paper, we generalize Theorem \ref{t11} in three directions:
\begin{enumerate}
\item from spin manifolds to oriented manifolds.
\item from circle actions to torus actions.
\item from odd type to generalized odd type.
\end{enumerate}

In this paper, for a torus action on a manifold, we introduce the notion of an action of generalized odd type.

\begin{definition} Let a $k$-torus $\mathbb{T}^k$ act on a manifold $M$. Let $S$ be a closed subgroup of $\mathbb{T}^k$.
\begin{enumerate}[(1)]
\item The $\mathbb{T}^k$-action is called \textbf{of even type with respect to $S$}, if for any connected component $Z$ of $M^S$, we have $\dim Z \equiv \dim M \mod 4$. An action of a torus $\mathbb{T}^k$ on a manifold is called \textbf{of generalized even type}, if it is of even type with respect to some closed subgroup of $\mathbb{T}^k$.
\item The $\mathbb{T}^k$-action is called \textbf{of odd type with respect to $S$}, if for any connected component $Z$ of $M^S$, we have $\dim Z \equiv \dim M-2 \mod 4$. An action of a torus $\mathbb{T}^k$ on a manifold is called \textbf{of generalized odd type}, if it is of odd type with respect to some closed subgroup of $\mathbb{T}^k$.
\end{enumerate}
\end{definition}

Therefore, a circle action on a manifold being of odd type is a special of a torus action of generalized odd type where the torus is the circle, and the closed subgroup is $\mathbb{Z}_2$. The main result of this paper is the following:

\begin{theorem} \label{t13} Let a torus act on a closed oriented manifold $M$. If the action is of generalized odd type, then the signature of $M$ vanishes. \end{theorem}

In \cite{HH2}, H. Herrera and R. Herrera prove vanishing results on characteristic classes of manifolds admitting circle actions. One of them is that if a $4n$-dimensional oriented manifold with finite second homotopy group admits a circle action of odd type, then its signature vanishes. As a special case of Theorem \ref{t13} where a torus group is the circle group and the closed subgroup is $\mathbb{Z}_2$, we recover the result without the assumption on the second homotopy group of the manifold.

\section{Preliminaries and the proof of the main result}

Let the circle act on a closed oriented manifold $M$. Then the equivariant index of the signature operator is defined for each element of $S^1$. In \cite{AS}, it is proved that the equivariant index is rigid under the circle action, i.e., is independent of the choice of an element of $S^1$, and is equal to the signature of $M$. Consequently, the signature of $M$ is equal to the sum of the signatures of the connected components of $M^{S^1}$.

Now, let a $k$-torus $\mathbb{T}^k$ act on a closed oriented manifold $M$. Then there exists a circle $S^1$ inside $\mathbb{T}^k$ that has the same fixed point set as $\mathbb{T}^k$, i.e., $M^{S^1}=M^{\mathbb{T}^k}$. As for $S^1$-actions, for an action of a torus the signature of $M$ is equal to the sum of the signatures of the connected components of $M^{\mathbb{T}^k}$. It follows from Theorem 6.12 in \cite{AS} and is stated explicitly in \cite{KR}.

\begin{theorem} \cite{AS}, \cite{KR} \label{t21} Let a $k$-torus $\mathbb{T}^k$ act on a closed oriented manifold $M$. Then $\textrm{sign}(M)=\textrm{sign}(M^{\mathbb{T}^k})$. \end{theorem}

By $\textrm{sign}(M^{\mathbb{T}^k})$, it means the sum of the signatures of all connected components of $M^{\mathbb{T}^k}$, i.e., $\displaystyle \textrm{sign}(M^{\mathbb{T}^k})=\sum_{N \subset M^{\mathbb{T}^k}} \textrm{sign}(N)$.

In \cite{Ko}, Kobayashi proves that the fixed point set of a torus action on an orientable manifold is orientable.

\begin{lemma} \label{l22} \cite{Ko} Let a torus act on an orientable manifold $M$. Then the fixed point set is a union of closed orientable manifolds. \end{lemma}

Given a circle action on an oriented manifold, H. Herrera and R. Herrera prove the orientability of the set of points fixed by a subgroup of $S^1$.

\begin{lemma} \label{l23} \cite{HH} Let $M$ be an oriented, $2n$-dimensional, smooth manifold endowed with a smooth $S^1$-action. Consider $\mathbb{Z}_k \subset S^1$ and its corresponding action on $M$. If $k$ is odd then the fixed point set $M^{\mathbb{Z}_k}$ of the $\mathbb{Z}_k$-action is orientable. If $k$ is even and a connected component $Z$ of $M^{\mathbb{Z}_k}$ contains a fixed point of the $S^1$-action, then $Z$ is orientable. \end{lemma}

Let $S$ be a closed subgroup of $\mathbb{T}^k$. Then $S$ is Lie isomorphic to a product of $S^1$'s and $\mathbb{Z}_a$'s, i.e., $S \approx S^1 \times \cdots \times S^1 \times \mathbb{Z}_{a_1} \times \cdots \times \mathbb{Z}_{a_m}$, where $a_i$'s are positive integers bigger than 1. Note that the $a_i$'s may have repeated elements. By using Lemma \ref{l22} and Lemma \ref{l23}, we extend Lemma \ref{l23} to torus actions.

\begin{lemma} \label{l24} Let a $k$-torus $\mathbb{T}^k$ act on a $2n$-dimensional orientable manifold $M$ and $S$ a closed subgroup of $\mathbb{T}^k$. Let $Z$ be a connected component of $M^S$. If $Z$ contains a $\mathbb{T}^k$-fixed point (i.e., if $Z \cap M^{\mathbb{T}^k} \neq \emptyset$), then $Z$ is orientable. \end{lemma}

\begin{proof} Without loss of generality, by choosing an orientation, assume that $M$ is oriented. Since $S$ is a closed subgroup of $\mathbb{T}^k$, $S$ is isomorphic to \begin{center} $S=(S^1)^l \times \mathbb{Z}_{a_1} \times \cdots \times \mathbb{Z}_{a_m}$ \end{center} for some $l \geq 0$ and positive integers $a_i>1$ for $i=1,\cdots,m$.

Denote $S_i=(S^1)^l \times \mathbb{Z}_{a_1} \times \cdots \times \mathbb{Z}_{a_i}$ for $i=0,1,\cdots,m$. Also, denote $M_i$ by the set of points fixed by $S_i$-action, i.e., $M_{i}=M^{S_{i}}$. We prove that for any $i$ if $Z_i$ is a connected component of $M_i$ that contains $Z$, then $Z_i$ is orientable.

Consider the case that $i=0$. Then $S_0=(S^1)^l$. By Lemma \ref{l22}, the set $M^{S_0}=M^{(S^1)^l}$ of points fixed by the $(S^1)^l$-action is a union of smaller dimensional closed orientable manifolds.

Suppose that a connected component $Z_{i-1}$ of $M_{i-1}$ is orientable and contains $Z$. On $Z_{i-1}$, we have an induced action of $G_i=\mathbb{T}^k/S_i = (S^1)^k / (S^1)^l \times \mathbb{Z}_{a_1} \times \cdots \times \mathbb{Z}_{a_{i}}=(S^1)^{k-l}=\mathbb{T}^{k-l}$. Moreover, as a subgroup of $G_i$, $\mathbb{Z}_{a_i}$ acts on $Z_{i-1}$. Given a generator $b_i$ of $\mathbb{Z}_{a_i}$, there exists $X_i$ in the Lie algebra of $\mathbb{T}^{k-l}$ such that $\mathrm{exp}(\frac{1}{a_i}X_i)=b_i$. In other words, $X_i$ generates a circle $S^1$ such that $\mathbb{Z}_{a_i} \subset S^1$. Denote the circle by $H_i$.

Therefore, on the orientable manifold $Z_{i-1}$, we have the action of the circle $H_i$ and the action of $\mathbb{Z}_{a_i}$, as a subgroup of $H_i$. A connected component $Z_i$, the set of points in $Z_{i-1}$ that are fixed by the $\mathbb{Z}_{a_i}$-action and contains $Z$, contains a $H_i$-fixed point, since it contains $Z$ which is fixed by the $\mathbb{T}^k$-action and $H_i$ is a subgroup of $\mathbb{T}^k$.

Apply Lemma \ref{l23} for the action of the circle $H_i$ on $Z_{i-1}$ with its subgroup $\mathbb{Z}_{a_i}$. Since the connected component $Z_i$ of $Z_{i-1}^{\mathbb{Z}_{a_i}}$ contains a $H_i$-fixed point, it follows that $Z_i$ is orientable.

Note that $Z_m=Z$. The lemma then follows by inductive argument. \end{proof}

Note that in Lemma \ref{l24}, we can remove the condition that $Z$ contains a $\mathbb{T}^k$-fixed point, if all $a_i$ are odd. With Lemma \ref{l24}, we are ready to prove our main result.

\begin{proof}[Proof of Theorem \ref{t13}] If the dimension of $M$ is not divisible by 4, then the signature of $M$ is defined to be 0 and hence the theorem follows. Therefore, from now on, suppose that the dimension of $M$ is divisible by 4.

Let $\mathbb{T}^k$ be the torus which acts on $M$. Let $S$ be the subgroup of the torus $\mathbb{T}^k$ such that the torus action is of odd type with respect to $S$. Let $Z$ be a connected component of the set $M^S$ of points fixed by the $S$-action on $M$ that contains a $\mathbb{T}^k$-fixed point, i.e., $Z \cap M^{\mathbb{T}^k} \neq \emptyset$. Then by Lemma \ref{l24}, $Z$ is orientable. Choose an orientation of $Z$. Since the $\mathbb{T}^k$-action is of odd type with respect to the closed subgroup $S$, $\dim Z \equiv \dim M -2 \mod 4$. Since $\dim M \equiv 0 \mod 4$, we have that $\dim Z \equiv 2 \mod 4$. Therefore, we have that $\textrm{sign}(Z)=0$. On $Z$, there is an induced action of $\mathbb{T}^k/S=\mathbb{T}^{k'}$. Moreover, the set of points in $Z$ that are fixed by the induced $\mathbb{T}^{k'}$-action is precisely the set of points in $Z$ that are fixed by the entire $\mathbb{T}^k$-action on $M$, i.e., $Z^{\mathbb{T}^{k'}}=Z \cap M^{\mathbb{T}^k}$. Applying Theorem \ref{t21} to the induced action of the torus $\mathbb{T}^{k'} =\mathbb{T}^k /S$ on $Z$, we have that 
\begin{center}
$\displaystyle \sum_{N \subset Z^{\mathbb{T}^{k'}}} \textrm{sign}(N)=\textrm{sign}(Z)=0$.
\end{center}

On the other hand, by directly applying Theorem \ref{t21} to the action of the $k$-torus $\mathbb{T}^k$ on $M$, we have that
\begin{center}
$\displaystyle \textrm{sign}(M)=\sum_{N \subset M^{\mathbb{T}_k}} \textrm{sign}(N)$.
\end{center}

Since $M^{\mathbb{T}^k} \subset M^S \subset M$, each connected component $N$ of $M^{\mathbb{T}^k}$ is contained in a unique connected component $Z$ of $M^S$ which contains a $\mathbb{T}^k$-fixed point, as $Z$ contains the $\mathbb{T}^k$-fixed component $N$. Therefore, we have that
\begin{center}
$\displaystyle \textrm{sign}(M)=\sum_{N \subset M^{\mathbb{T}_k}} \textrm{sign}(N)=\sum_{Z \subset M^S, Z \cap M^{\mathbb{T}^k} \neq \emptyset} \sum_{N \subset Z \cap M^{\mathbb{T}_k}} \textrm{sign}(N)$

$\displaystyle =\sum_{Z \subset M^S, Z \cap M^{\mathbb{T}^k} \neq \emptyset} \textrm{sign}(Z)=0$.
\end{center}
\end{proof}

\end{document}